\documentclass[12pt, reqno]{amsart}
\usepackage{amsmath, amsthm, amscd, amsfonts, amssymb, graphicx, color, mathrsfs}
\usepackage[bookmarksnumbered, colorlinks, plainpages]{hyperref}
\usepackage[all]{xy}
\usepackage{slashed}

\newtheorem{theorem}{Theorem}[section]
\newtheorem{lemma}[theorem]{Lemma}

\newtheorem{corollary}[theorem]{Corollary}
\theoremstyle{definition}
\newtheorem{definition}[theorem]{Definition}

\theoremstyle{remark}
\newtheorem{remark}[theorem]{Remark}
\numberwithin{equation}{section}

\def\ind{{\mathcal I}}

\allowdisplaybreaks

\begin{document}
	\setcounter{page}{1}
	
\title[$L^p$-$L^q$ boundedness of Fourier multipliers]{$L^p$-$L^q$ boundedness of Fourier multipliers associated with the anharmonic Oscillator}
	
	\author[Marianna Chatzakou]{Marianna Chatzakou}
	\address{
		Marianna Chatzakou:
		\endgraf
		Department of Mathematics
		\endgraf
		Imperial College London
		\endgraf
		180 Queen's Gate, London SW7 2AZ
        \endgraf
         United Kingdom
		\endgraf
		{\it E-mail address} {\rm m.chatzakou16@imperial.ac.uk}
		}
	
	\author[Vishvesh Kumar]{Vishvesh Kumar}
	\address{
		Vishvesh Kumar:
		\endgraf
		Department of Mathematics: Analysis Logic and Discrete Mathematics
		\endgraf
		Ghent University, Belgium
		\endgraf
		{\it E-mail address} {\rm vishveshmishra@gmail.com}
		\endgraf
	}

\thanks{The authors are grateful to Prof. Micheal Ruzhansky for valuable suggestions. Marianna Chatzakou is supported by the EPSRC grant EP/N509486/1. Vishvesh Kumar is supported by the FWO Odysseus 1 grant G.0H94.18N: Analysis and Partial Differential Equations.}


	
	
	\date{Received: xxxxxx;  Revised: yyyyyy; Accepted: zzzzzz.
	}
	
	\begin{abstract} In this paper we study the $L^p$-$L^q$ boundedness of the Fourier multipliers in the setting where the underlying Fourier analysis is introduced with respect to the eigenfunctions of an anharmonic oscillator $A$. Using the notion of a global symbol that arises from this analysis, we extend a version of the Hausdorff-Young-Paley inequality that guarantees the $L^p$-$L^q$ boundedness of these operators for the range $1<p \leq 2 \leq q <\infty$. The boundedness results for spectral multipliers acquired, yield as particular cases Sobolev embedding theorems and time asymptotics for the $L^p$-$L^q$ norms of the heat kernel associated with the anharmonic oscillator. Additionally, we consider functions $f(A)$ of the anharmonic oscillator on modulation spaces and prove that Linsk\u ii's trace formula holds true even when $f(A)$ is simply a nuclear operator.   \\
	\end{abstract} \maketitle
	
	\tableofcontents
	
\section{Introduction}
The main aim of this paper is to establish sufficient conditions for the $L^p$-$L^q$ boundedness of the Fourier multipliers associated with a family of anharmonic oscillators on $\mathbb{R}^n$ as considered in \cite{CDR18}. The boundedness of  Fourier multipliers is a central theme of harmonic analysis with far reaching applications in different areas. Such analysis is traced back to H\"ormander's seminal paper \cite{Hormander1960} in 1960. Later, the $L^p$-boundedness of Fourier multipliers and spectral multipliers was thoroughly investigated by several researchers in many different settings, we cite here \cite{Anker, Anker92, CR, Ruzwirth, BT, CR16} to mention a very  few of them. 
\par Here we deal with $L^p$-$L^q$ multipliers as opposed to the $L^p$-multipliers, for which theorems of Mihlin-H$\ddot{\text{o}}$rmander or Marcinkiewicz type, provide results in different settings based on the regularity of the symbol. The $L^p$-$L^q$ boundedness of Fourier multipliers and spectral multipliers on locally compact unimodular groups and on compact homogeneous spaces is recently settled by Akylzhanov et al. \cite{AR,ARN, ARN1}. 

Anharmonic oscillators on $\mathbb{R}^n$ are important operators in analysis and mathematical physics, but also in number theory. The analysis of the energy levels of the Schr\"{o}dinger operator $i\partial_t \psi=-\Delta \psi+V(x)\psi$ can be reduced to the corresponding eigenvalue problem of an operator of the form $A=-\Delta+V(x)$; the so-called \textit{anharmonic oscillator}. \\
\indent Unlike the case of the harmonic oscillator (case $V(x)=|x|^2$) where the eigenfunctions are the Hermit functions and are well-understood (cf. \cite{Thangavelu}, \cite{Par10}), the exact solution of the eigenvalue problem associated with the anharmonic oscillator is still unknown. Despite the intensive research on the this topic over the last 40 years (cf. \cite{AM}, \cite{DJ}, \cite{HR1}, \cite{HR2}, \cite{HM}, \cite{V}), the exact solution of it, even in the quartic case, is still unknown \cite{OU}. The literature on the subject is vast and overviews can be found, for instance in \cite{CDR18}.\\
\indent Here we consider operators which in \cite{CDR18} has been regarded as ``prototype'' in the class under consideration. For such operators the potential is given as $V(x)=|x|^{2k}$, where $k \geq 1$ is an integer, and we generalise $A$ by considering higher order derivatives. In particular, for $l \geq 1$ integer, we define $A$ to be
\begin{equation}\label{defn.anh1}
A=A_{k,l}=(-\Delta)^{l}+|x|^{2k}+1\,.
\end{equation}
\indent The Weyl-H\"{o}rmander calculus associated with the anharmonic oscillator $A_{k,l}$ that has been developed in \cite{CDR18}, will be used in order to prove the continuous inclusion of the Sobolev spaces associated to $A_{k,l}$ (in the sense of Weyl-H\"{o}rmander calculus) to some suitable space, and finally get an estimate for the operator norm of Fourier multipliers. \\
\indent  To prove the above, one needs to follow classical techniques developed by H\"ormander \cite{Hormander1960}.
The Paley-type inequality \cite{HP} describes the growth of the Fourier transform of a function in terms of its $L^p$-norm. Interpolating the latter with the Hausdorff-Young inequality, one can obtain the following H\"ormander's version of the  Hausdorff-Young-Paley inequality,
\begin{equation}\label{3}
    \left(\int\limits_{\mathbb{R}^n}|(\mathscr{F}f)(\xi)\phi(\xi)^{ \frac{1}{r}-\frac{1}{p'} }|^r\right)^{\frac{1}{r}}\leq \Vert f \Vert_{L^p(\mathbb{R}^n)},\,\,\,1<p\leq r\leq p'<\infty, \,\,1<p<2,
\end{equation}
where $\mathscr{F}f$ stands for the Fourier transform of $f$ on $\mathbb{R}^n$ and $\phi$ is a positive function defined on $\mathbb{R}^n.$ As a consequence  of the Hausdorff-Young-Paley inequality, H\"ormander \cite[page 106]{Hormander1960} proves that the condition 
\begin{equation}\label{4}
    \sup_{t>0}t^b\{\xi\in \mathbb{R}^n:m(\xi)\geq t\}<\infty,\quad \frac{1}{p}-\frac{1}{q}=\frac{1}{b}\,,
\end{equation}where $1<p\leq 2\leq q<\infty,$ implies the existence of a bounded extension of a Fourier multiplier $T_{m}$ with symbol $m$ from $L^p(\mathbb{R}^n)$ to $L^q(\mathbb{R}^n).$ Recently, Akylzhanov and Ruzhansky studied H\"ormander's classical results for unimodular locally compact groups  \cite{AR}. In their paper the key idea is the extension of H\"ormander's theorem to the unimodular locally compact  groups is the reformulation of this classical theorem in terms of 
noncommutative Lorentz spaces and group von-Neumann algebras of unimodular groups. 

The following reads as the Hausdorff-Young-Paley inequality in our setting where the appearing Fourier transform $\mathcal{F}_{A_{k,l}}$ is taken with respect to the eigenfunctions of $A_{k,l}$ precisely defined in Section \ref{Section3}. 
\begin{theorem}\label{thm12} Let $1<p\leq 2,$ and let   $1<p \leq b \leq p' \leq \infty,$ where $p'= \frac{p}{p-1}$.  If $\varphi(j)$ is a positive sequence in $\mathbb{N}$ such that 
 $$M_\varphi:= \sup_{t>0} t  \sum_{\overset{j \in \mathbb{N}}{t \leq \varphi(j) }}  \|u_j\|^2_{L^\infty(\mathbb{R}^n)}    $$
is finite, then for every $f \in {L^p(\mathbb{R}^n)}$ 
 we have
\begin{equation} \label{eq1.3}
    \left( \sum_{j \in \mathbb{N}}  \left( |\mathcal{F}_{A_{k, l}}f(j)| \varphi(j)^{\frac{1}{b}-\frac{1}{p'}} \right)^b \|u_j\|_{L^\infty(\mathbb{R}^n)}^{2-b}       \right)^{\frac{1}{b}} \lesssim_p M_\varphi^{\frac{1}{b}-\frac{1}{p'}} \|f\|_{L^p(\mathbb{R}^n)}.
\end{equation}
\end{theorem}
Here, $(u_j)_{j \in \mathbb{N}}$ are the eigenfunctions of the anharmonic oscillator $A_{k,l}$ corresponding to the eigenvalues $(\lambda_j)_{j \in \mathbb{N}}$ in decreasing order. Although, the form of the eigenfuctions $u_j$ is unknown, the Sobolev estimates in this setting as in \cite{CDR18}, show that the $L^\infty$ norm of them is finite, and as a result \eqref{eq1.3} is well-defined; see Lemma \ref{Sob.emb}. For a complete exposition of the necessary notation and ideas involved in \eqref{eq1.3} we refer to Sections \ref{Section2} and \ref{Section3}.
\par The analysis on the asymptotic behaviour of the eigenvalues of the anharmonic oscillator $A_{k,l}$ as appeared in \cite{BBR}, or in more generality in \cite{CDR18}, together with Theorem \ref{thm12} and the analysis prior to this, leads to the following result on the $L^p-L^q$ boundedness of the Fourier multipliers associated with the anharmonic oscillator. 
\par We note that the following result is new even in the setting of the simpler case of the harmonic oscillator in any dimension. Precisely, a simplification of the arguments used here, would lead to an analogous to Theorem \ref{thm12} result in the setting of the harmonic oscillator where, of course, the Fourier analysis, should be regarded as the one associated with it.   
\begin{theorem} \label{Multiintro}
Let $1<p \leq 2 \leq q <\infty.$  Suppose that $m$ is a $A_{k,l}$-Fourier multiplier with $A_{k,l}$-symbol $\sigma_{m}$ on $\mathbb{R}^n.$ Then we have 
$$ \|m\|_{L^p(\mathbb{R}^n) \rightarrow L^q(\mathbb{R}^n)} \lesssim \sup_{s>0} s\left(  \ \sum_{\overset{j \in \mathbb{N}}{|\sigma_m(j)| >s}} \|u_j\|^2_{L^\infty(\mathbb{R}^n)}  \right)^{\frac{1}{p}-\frac{1}{q}}.$$
\end{theorem}
As an application of Theorem \ref{Multiintro} we get the following theorem on the boundedness of spectral multipliers of the aharmonic oscillator.

  \begin{corollary}\label{corintro}
  Let $1<p\leq 2 \leq q<\infty$, and let $\varphi$ be any decreasing function on $[1,\infty)$ such that
 $\lim_{u \rightarrow \infty} \varphi(u)=0$. Then,
  \begin{equation}
  \|\varphi(A_{k,l})\|_{L^p(\mathbb{R}^n) \rightarrow L^q(\mathbb{R}^n)} \lesssim \sup_{u>1} \varphi(u)\left(u^{1+\frac{(k+l)n}{2kl}}\right)^{\frac{1}{p}-\frac{1}{q}}. 
  \end{equation}
 \end{corollary}
 
Corollary \ref{corintro} yields time asymptotic for the $L^p$-$L^q$ norms of heat kernel associated with the anharmonic oscillator $A_{k,l}$; that is the fundamental solution of the heat equation when the Laplace operator is replaced by the differential operator $A_{k,l}$ as in Remark \ref{Heatintro}, and Sobolev-type estimates; see Remark \ref{Sobointro}. Finally, it allows to find the range of $m>0$ for which $A_{k,l}^{-m}$ is bounded from $L^{p}(\mathbb{R}^n)$ to $L^q(\mathbb{R}^n)$, see Corollary \ref{bound.inverse}

In the last part of the paper we study the $r$-nuclearity of functions of the anharmonic oscillator when acting on the modulation spaces $\mathcal{M}_{w}^{p,q}$ on $\mathbb{R}^n$. In this concept we are investigating the conditions on the decomposition of the operator that allow for the operator to be $r$-nuclear and also for Lindsk\u ii's trace formula (also called the nuclear trace) to hold in cases where $r \in (0,1]$ is not in the standard range $(0,2/3]$. These concepts have been studied by many authors in different settings; for instance we refer to \cite{BarazaD, CardonaFIO, CardonaBesov,CardonaKumar20181,CardonaKumar20182, CCK, Del, DR, DRT17, DW, KS, Toft1, Toft}.
The following is our main theorem in this context.      

\begin{theorem}
Let $0<r \leq 1$, $1\leq p,q<\infty$ and $w$ be a submultiplicative polynomially moderate weight. The operator $f(A_{k,l})$ is $r$-nuclear on $\mathcal{M}_{w}^{p,q}(\mathbb{R}^n)$, provided that 
	\begin{equation}\label{sec5.thm.as.1}
	\sum_{j=1}^{\infty}|f(\lambda_j)|^r \|u_j\|^{r}_{\mathcal{M}_{w}^{p,q}}\|u_{j}\|^{r}_{\mathcal{M}_{w^{-1}}^{p^{'},q^{'}}}<\infty\,.
	\end{equation}
	If, in particular, \eqref{sec5.thm.as.1} holds for $r=1$, then we have the trace formula
	\begin{equation}
	\label{sec5.thm.as.2}
	\textnormal{Tr}(f(A_{k,l}))= \sum_{j=1}^{\infty} f(\lambda_j)\,,
	\end{equation}
	where the series $\sum_{j=1}^{\infty} f(\lambda_j)$ converges absolutely.
\end{theorem}

\section{Basics of the anharmonic oscillator} \label{Section2}
Let us begin by recalling some basic facts about the anharmonic oscillator as in \cite{CDR18}, and performing some preliminary analysis on it. For the rest of this section, and the subsequent ones, we assume that the integers $k,l \geq 1$ in \eqref{defn.anh1} are fixed, i.e., we can write 
\begin{equation}\label{defn.anh}
    A=(-\Delta)^{l}+|x|^{2k}+1\,,\quad x \in \mathbb{R}^n\,,
\end{equation}
when referring to the anharmonic oscillator, where $\Delta$ stands for the Laplace-Betrami operator, and $|\cdot|$ for the euclidean norm on $\mathbb{R}^n$. Let us note that the constant in \eqref{defn.anh} allows for the invertibility of $A$.
\par If $A$ is viewed as a pseudo-differential operator, then it arises as the Weyl-quantization of the symbol $a(x,\xi)=|\xi|^{2l}+|x|^{2k}+1$, where $x,\xi \in \mathbb{R}^n$, and we write $A=a^w$ to denote this relation. In \cite{CDR18} the authors proved that, in the setting of Weyl-H\"{o}rmander calculus, we have
\[
a \in S(M,G)\,,
\]
where the H\"{o}rmander metric $G$ is given by 
\begin{equation}\label{metric}
G(dx,d\xi)=\frac{dx^2}{(|\xi|^{2l}+|x|^{2k}+1)^{\frac{1}{k}}}+\frac{d\xi^2}{(|\xi|^{2l}+|x|^{2k}+1)^{\frac{1}{l}}}\,,
\end{equation}
and the $G$-weight $M$ is given by 
\begin{equation}\label{weight}
    M(x,\xi)=|\xi|^{2l}+|x|^{2k}+1\,.
\end{equation}

\indent It is well-known from the general theory of pseudo-differential operators (c.f \cite[Section 1.7]{NiRo:10}) that if $\sigma^w$ is an injective pseudo-differential operator on the Schwartz space $\mathcal{S}(\mathbb{R}^n)$, whose symbol $\sigma \in S(m,g)$ is real and elliptic, then $\sigma^w$ is an isomorphism on $\mathcal{S}(\mathbb{R}^n)$ and on $\mathcal{S}^{'}(\mathbb{R}^n)$; the Schwartz space and its dual, and moreover, its inverse is a pseudo-differential operator with real and elliptic Weyl-symbol in the class $S(m^{-1},g)$. We recall here that a symbol $\sigma \in S(m,g)$ is \textit{elliptic} if for some $R>0$ 
\[
\sigma (x,\xi) \geq C m(x,\xi)\,,\quad \text{for}\quad x,\xi >R\,, C>0\,.
\]
\indent Simple calculations show that $A$ is an injective operator on $\mathcal{S}(\mathbb{R}^n)$, while one can check that its symbol is elliptic in the class $S(M,G)$ defined above. Hence, its inverse $A^{-1}$ is a positive, and also a compact operator since $M^{-1}(x,\xi) \rightarrow \infty$ as $x, \xi \rightarrow 0$; cf. \cite[Theorem 1.4.2]{NiRo:10}. 

\indent The operator $A$ admits a unique self-adjoint extension on $L^2(\mathbb{R}^n)$, and its spectrum is purely discrete and lies in $(0,\infty)$. By the spectral theorem there exists an orthonormal basis, say $(u_j)_{j \in \mathbb{N}}$, of $L^2(\mathbb{R}^n)$ made of eigenfunctions of it. Let $\left(\lambda_j \right)_{j \in \mathbb{N}}$ be the corresponding eigenvalues.
\par The functional calculus allows for the definition of the operator $A^m$, for each $m \in \mathbb{R}$, with domain
\begin{equation}\label{dom.Ar}
\textnormal{Dom}\left(A^{m}\right)=\{u \in L^2(\mathbb{R}^n) : \sum_{j=1}^{\infty} \lambda_{j}^{2m} |\langle u_j,u \rangle_{L^2}|^2 <\infty \}\,,
\end{equation}
so that if $u \in \textnormal{Dom} \left(A^{m}\right)$, then we can write
\[
A^{m}u=\sum_{j=1}^{\infty} \lambda_{j}^{m} \langle u_{j},u \rangle_{L^2} u_j\,.
\]

\par In the context of the anharmonic oscillator $A$ the below Sobolev spaces are introduced in \cite[Definition 4.7]{CDR18}: For fixed $k,l \geq 1$ integers, and for $m \in \mathbb{R}$, we denote by $\mathcal{H}^{m}_{k,l}$, or simply by $\mathcal{H}^{m}$, the set of tempered distributions $u$ such that \[
A^{\frac{m}{2}}u \in L^2(\mathbb{R}^n)\,,
\]
and for $u \in \mathcal{H}^m$ we define 
\begin{equation}\label{Sob.spaces.anh}
    \|u\|_{\mathcal{H}^m}:=\|A^{\frac{m}{2}}u\|_{L^2(\mathbb{R}^n)}\,.
\end{equation}
The Sobolev spaces $\mathcal{H}^{m}$ reads as
\begin{equation*}
\mathcal{H}^m \equiv H(M^{\frac{m}{2}},G)\,,
\end{equation*}
 with regards to the notion of Sobolev spaces adjusted to the Weyl-H\"{o}rmander calculus, where $M,G$ are as in \eqref{weight} and \eqref{metric}, respectively. 
 
 Recall that for some H\"ormander metric $g$ and some $g$-weight $m$, the Sobolev space denoted by $H(m,g)$ is the set of tempered distributions $u$ such that 
\[
a^{w}u \in L^2(\mathbb{R}^n)\,,\quad \forall a \in S(m,g)\,,
\]
where $a^w$ denotes the Weyl-quantization of the symbol $a$.
\par For the Sobolev embeddings in the scale of Weyl-H\"ormander calculus, we have the following result due to Chemin and Xu \cite[Theorem 1.9]{CX}.
\begin{theorem}
    \label{thm.Xu}
    Let $m_1$, $m_2$ be two $g$-weights such that 
    \[
    \lim_{x,\,\xi \rightarrow \infty}\frac{m_1(x,\xi)}{m_2(x,\xi)}=\infty\,,
    \]
    then $H(m_1,g)$ is compactly included in $H(m_2,g)$.
\end{theorem}
The next technical lemma will be useful for our purposes.
\begin{lemma}\label{Sob.emb}
Let $(u_j)_{j \in \mathbb{N}}$ be the set of eigenvalues of $A$ with corresponding eigenvalues $(\lambda_j)_{j \in \mathbb{N}}$. There exists $C>0$ such that $\|u_j\|_{L^{\infty}(\mathbb{R}^n)}\leq C \sqrt{\lambda_j}$ for every $j \in \mathbb{N}$.
\end{lemma}

\begin{proof}
	It is known, cf. \cite[Theorem 8.8]{BR}, that the usual Sobolev space of order $p$; that is the space $H^{p}(\mathbb{R}^n)=H(\langle \xi \rangle^{p},g^{1,0})$, where $g^{1,0}$ is defined by
	\[
	g^{1,0}(dx,d\xi)=|dx|^2+|d\xi|^2/\langle \xi \rangle^{-2}\,,
	\]
	is continuously included in $L^{\infty}(\mathbb{R}^n)$ for any $p \in [1,\infty]$. On the other hand, for $m \in \mathbb{R}$, we have $H(M^{\frac{m}{2}},g^{1,0})$, where $M$ as in \eqref{weight}, is compactly included in $H^{p}(\mathbb{R}^n)$ for $p \leq lm$, where $k,l \geq 1$ are the fixed integers associated to $A$. Indeed the last is true by Theorem \ref{thm.Xu}, since
	\[
	\underset{x,\xi \rightarrow \infty}{\lim} \frac{(1+|x|^{2k}+|\xi|^{2l})^{\frac{m}{2}}}{\langle \xi \rangle^{p}}=\infty\,,\quad \text{for}\quad lm\geq p\,.
	\]
	For $g$ as in \eqref{metric} we have the inclusion $H(M^{\frac{m}{2}},g) \subset H(M^{\frac{m}{2}},g^{1,0})$. Summarising the above, and choosing $p=1$ we conclude that the space $H(M^{\frac{m}{2}},g^{k,l})$ is continuously embedded in the space $L^{\infty}(\mathbb{R}^n)$ under the condition $lm \geq 1$. The latter implies that
	\[
	\|u_j\|_{L^{\infty}(\mathbb{R}^n)}\leq C\|A^{\frac{m}{2}}u_j\|_{L^{2}(\mathbb{R}^n)}\leq C \lambda^{\frac{m}{2}}_{j} \|u\|_{L^{2}(\mathbb{R}^n)}=C \lambda^{\frac{m}{2}}_{j}\,,\quad C>0\,,
	\]
	for every $m \in\mathbb{R}$ such that $lm \geq 1$, and so also for $m=1$ and this completes the proof of Lemma \ref{Sob.emb}.
\end{proof}

\indent Finally, let us briefly recall some spectral properties of the anharmonic oscillator $A$. In \cite[Corollary 5.6]{CDR18} that authors prove that $A^{-1} \in S_r(L^2(\mathbb{R}^n))$ for $r>\frac{(k+l)n}{2kl}$, where by $S_r(L^2(\mathbb{R}^n))$ we have denoted the \textit{$r$th Schatten-von Neumann class} of operators on the Hilbert space $L^2(\mathbb{R}^n)$. For an operator $T$ in the class $S_r$, with singular values $(s_j(T))_j$\footnote{The singular values of a compact operator $T$ between Hilbert spaces are the square roots of non-negative eigenvalues of the self-adjoint operator $T^{*}T$, where $T^{*}$ denotes the adjoint of $T$.}, it is known that 
\[
s_j(T)=o(j^{-\frac{1}{r}})\,,\quad \text{as}\quad j \rightarrow \infty\,,
\]
where the singular values appear in decreasing order. Hence, in the particular case of the compact, self-adjoint $A^{-1}$, in which we have $s_{j}=\lambda_{j}$, the above property reads as follows 
\begin{equation}\label{eigen.vanish}
\lambda_{j}^{-1}=o(j^{\frac{1}{r}})\,,\quad \text{as}\quad j \rightarrow \infty\,,
\end{equation}
for $r>\frac{(k+l)n}{2kl}$ where the eigenvalues $(\lambda_{j}^{-1})_j$ appear in decreasing order. Moreover, if $N(\Lambda)$ is the eigenvalue counting function associated with the anharmonic oscillator $A$; that is
\[
N(\Lambda):=\{\#j : \lambda_j\leq \Lambda \}\,,
\]then we have the following estimate 
\begin{equation}
    \label{eig.count.funt.anhar.}
    N(\Lambda) \lesssim \Lambda^{n\left(\frac{1}{2k}+\frac{1}{2l}\right)}\,,\quad \text{as}\quad \Lambda \rightarrow \infty\,,
\end{equation}
cf. \cite[Theorem 3.2]{BBR}.

\section{Fourier analysis associated with the anharmonic oscillator} \label{Section3}
In this section we present the Fourier analysis associated with the eigenfunction expansion of the anharmonic oscillator $A$  as in \eqref{defn.anh}. The theory presented here is a natural analogue of Fourier analysis developed in \cite{Ruz-Tok} in the setting of  nonharmonic analysis associated with a model operator. Therefore we will not present any proofs for any of our statements which can be derived verbatim as in \cite{Ruz-Tok}. 

As before, we denote by $(u_j)_{j \in \mathbb{N}}$ the eigenvalues of the operator $A$, and by $(\lambda_j)_{j \in \mathbb{N}}$ the corresponding eigenvalues in increasing order. \\

Let us now define the space of functions
\begin{equation}\label{Eq:R-T-dom}
C^\infty_{A}(\mathbb{R}^n):=\cap_{m=1}^\infty \textnormal{Dom}(A^m)\,,
\end{equation}
where the domain of $A^m$ is as in \eqref{dom.Ar}, whose Fr\'echet topology is given by the family of norms
$$\|f\|_{C^m_{A}}:=\max_{n\leq m} \| A^n f\|_{L^2(\mathbb{R}^n)},\,\, m\in\mathbb{N}_0,\,\,f\in  C^\infty_{A}(\mathbb{R}^n).$$

Since $(u_j)_{j \in \mathbb{N}}$ is dense in $L^2(\mathbb{R}^n)$ we have  that $C^\infty_{A}(\mathbb{R}^n)$ is dense in $L^2(\mathbb{R}^n)$.\\
\\
\textbf{$A$-Fourier transform:}
Let $\mathcal{S}(\mathbb{N})$ be the space of rapidly decreasing functions $\phi:\mathbb{N}\to \mathbb{C}$, i.e., for any $N<\infty,$ there exists a constant $C_{\phi,N}$ such that $|\phi(\xi)|\leq C_{\phi,N}\langle \xi\rangle^{-N}\,\,\text{for all } \xi\in\mathbb{N}.$ The space  $\mathcal{S}(\mathbb{N})$ forms a Fr\'echet space if endowed with the family $\{p_r\}_{r \in \mathbb{N}}$ of semi-norms \[p_r(\phi):=\sup_{\xi\in\mathbb{N}} \langle\xi\rangle^r |\phi(\xi)|\,.\]
We define the \textit{$A$-Fourier transform} to be that bijective homeomorphism \[\mathcal{F}_{A}:C^\infty_{A}(\mathbb{R}^n)\to  \mathcal{S}(\mathbb{N})\] defined by 
\begin{equation}\label{Lfourier}
(\mathcal{F}_{A} f)(\xi):=\hat{f}(\xi):=\int\limits_{\mathbb{R}^n}  f(x)\overline{u_j(x)}\, dx.
\end{equation}
The  inverse operator $\mathcal{F}^{-1}_{A}: \mathcal{S}(\mathbb{N})\to  C^\infty_{A}( \mathbb{R}^n)$ is given by  $$(\mathcal{F}^{-1}_{A} h)(x):=\sum_{j \in \mathbb{N}} h(\xi) u_j(x)\,,$$ so that the Fourier inversion formula becomes
\begin{equation}
f(x)=\sum_{j \in \mathbb{N}} \hat{f}(\xi) u_j(x),\,\, f\in C^\infty_{A}(\mathbb{R}^n).
\end{equation} 

The Plancherel identity holds true for the $A$-Fourier transform of functions in the space $l^{2}_{A}$ that we define as the (linear) space of complex-valued sequences $a$ on $\mathbb{N}$ such that $\mathcal{F}^{-1}_{A}a \in L^2(\mathbb{R}^n)$. The latter means that there exists $f \in L^2(\mathbb{R}^n)$ such that $\mathcal{F}_{A}f=a$.
The space of sequences  $l^2_{A}$ is a Hilbert space under the norm \[\|a\|=\left(\sum_{j \in \mathbb{N}}|a(j)|^2\right)^{\frac{1}{2}}\,,\] since by formal calculations
$$(a,b)_{l^2_{\mathcal{A}}}=(\mathcal{F}^{-1}_{\mathcal{A}} a,\mathcal{F}^{-1}_{\mathcal{A}} b )_{L^2}\,,\quad a,b \in l^{2}_{\mathcal{A}}\,,$$ and the inner product is then given by
$$(a,b)_{l^{2}_{A}}:= \sum_{j \in\mathbb{N}} a(j)\overline{b(j)} \,,$$
for arbitrary $a,b \in l^{2}_{A}$. 
	For any $f \in L^2(\mathbb{R}^n)$ we have $\widehat{f} \in l^{2}_{A}$, and the Plancherel identity is satisfied
	\begin{equation}	\label{planch.A.ft}
	\|f\|_{L^2}=\|\widehat{f}\|_{l^{2}_{A}}\,.
	\end{equation}
Let us now define the generalisation of the space $l^{2}_{A}$ to be the space of functionals $a$ on $\mathcal{S}(\mathbb{N})$ denoted by $l^{p}_{A}$, for $1\leq p <\infty$, such that  
	\begin{equation}\label{EQ:norm1}
	\|a\|_{l^{p}({A})}:=\left(\sum_{j \in \mathbb{N}}| a(j)|^{p}
	\|u_j\|^{2-p}_{L^{\infty}(\mathbb{R}^n)} \right)^{1/p}<\infty\,.
	\end{equation}
	The space $l^{\infty}_{A}$ shall be defined as the set of $a \in \mathcal{S}^{'}(\mathbb{N})$ such that 
	$$
	\|a\|_{l^{\infty}(A)}:=\sup_{j \in\ind}\left( |a(j)|\cdot
	\|u_j\|^{-1}_{L^{\infty}(\mathbb{R}^n)}\right)<\infty.
	$$

The following theorem reads as the Hausdorff-Young inequality in our setting, and can be proved using standard interpolation properties as in [Theorem 7.6 \cite{Ruz-Tok}].  

\begin{theorem}[Hausdorff-Young inequality]\label{haus} Let $1 \leq p \leq 2$ and let $\frac{1}{p}+\frac{1}{p'}=1.$ Then, there is a constant $C_p\geq 1$ such that, for all $f \in L^p(\mathbb{R}^n)$ and for all $a \in l^{p}(A)$, we have 
\begin{equation}\label{haus.1}
\|\widehat{f}\|_{l^{p^{'}}(A)}\leq C_p \|f\|_{L^p}\,,
\end{equation}
and
\[
\|\mathcal{F}^{-1}_{A}a\|_{L^{p^{'}}}\leq C_p \|a\|_{l^{p}(A)}\,.
\]
\end{theorem}

\section{$L^p$-$L^q$ boundedness of Fourier multipliers  for $1<p\leq 2\leq q<\infty$}\label{LpLq}
In this section we present our boundedness results, first in a general form, and later on we specialise by considering particular examples of operators. Our analysis begins with proving our version of two well-known inequalities; namely, the Paley-type inequality, and the Hausdorff-Young-Paley-type inequality.

\subsection{Hausdorff-Young-Paley inequality}  

The Hausdorff-Young-Paley inequality is a generalisation of the Paley-type inequality that follows. For the proof of it interpolation techniques as in Corollary \ref{interpolationoperator} are necessary.

\begin{theorem}[Paley-type inequality]\label{PI}
Let  $1<p \leq 2$.  If $\varphi(j)$ is a positive sequence in $\mathbb{N}$ such that 
 $$M_\varphi:= \sup_{t>0} t  \sum_{\overset{j \in \mathbb{N}}{t \leq \varphi(\xi) }}  \|u_j\|_{L^\infty(\mathbb{R}^n)}^2  <\infty,  $$
 then for every $f \in {L^p(\mathbb{R}^n)}$ we have
\begin{equation} \label{Vish5.1}
    \left( \sum_{j \in \mathbb{N}} |\mathcal{F}_{A}(f)(j)|^p \|u_j\|_{L^\infty(\mathbb{R}^n)}^{2-p}  \varphi(j)^{2-p}   \right)^{\frac{1}{p}} \lesssim M_\varphi^{\frac{2-p}{p}} \|f\|_{L^p(\mathbb{R}^n)}.
\end{equation}
\end{theorem} 
\begin{proof} 	Let $\nu$ be the measure on $\mathbb{N}$ defined by $\nu(\{j\}):= \varphi^2(j) \|u_j \|_{L^\infty(\mathbb{R}^n)}^2 $ for $j \in \mathbb{N}.$ Now, we define weighted spaces $L^p(\mathbb{N}, \nu),$ $1 \leq p \leq 2,$ as the spaces of complex (or real) sequences $a=\{a_j\}_{j \in \mathbb{N}}$ such that 
	\begin{equation}
	\|a\|_{L^p(\mathbb{N}, \nu)}:= \left( \sum_{j \in \mathbb{N}} |a_j|^p \varphi^2(j)\, \|u_j\|_{L^\infty(\mathbb{R}^n)}^2  \right)^\frac{1}{p}<\infty.
	\end{equation}
	We show that the subadditive operator 
	$T:L^p(\mathbb{R}^n): \rightarrow L^p(\mathbb{N}, \nu)$ defined by  $$ Tf:= \left\{ \frac{|\mathcal{F}_{A}(f)(j)|} {\|u_j\|_{L^\infty(\mathbb{R}^n)}  \varphi(j)} \right\}_{j \in \mathbb{N}}$$ is well-defined and bounded from $L^p(\mathbb{R}^n)$ to $L^p(\mathbb{N}, \nu)$ for $1<p \leq 2.$ In other words, we claim that we have the estimate 
	\begin{align} \label{vish5.3}
	\| Tf \|_{L^p(\mathbb{N}, \,\nu)} &= \left( \sum_{j \in \mathbb{N}} \left( \frac{|\mathcal{F}_{A}(f)(j)|}{\|u_j\|_{L^\infty(\mathbb{R}^n)}  \varphi(j)} \right)^p \varphi^2(j) \|u_j \|_{L^\infty(\mathbb{R}^n)}^2 \right)^\frac{1}{p} \nonumber \\& \lesssim M_\varphi^{\frac{2-p}{p}} \|f\|_{L^p(\mathbb{R}^n)},
	\end{align} 
	which would give us \eqref{Vish5.1} and where we set $$M_\varphi:= \sup_{t>0} t \sum_{\overset{j \in \mathbb{N}}{   \varphi(j) \geq t}} \|u_j\|_{L^\infty(\mathbb{R}^n)}^2 .$$ To prove this we will show that $T$ is of weak type $(2,2)$ and of weak type $(1,1).$ More precisely, for the distribution function,  
	$$\nu_{\mathbb{N}}(y; Tf)= \sum_{\overset{j \in \mathbb{N}}{|Tf (j)|\geq  y}} \|u_j\|_{L^\infty(\mathbb{R}^n)}^2   \varphi^2(j)\,,\quad \text{where}\quad y \in \mathbb{N}$$
	we show that 
	\begin{equation} \label{vish5.4}
	\nu_{\mathbb{N}}(y; Tf) \leq \left( \frac{M_2 \|f\|_{L^2(\mathbb{R}^n)}}{y} \right)^2 \,\,\,\,\,\,\text{with norm}\,\, M_2=1,
	\end{equation}
	and 
	\begin{equation} \label{vish5.5}
	\nu_{\mathbb{N}}(y; Tf) \leq \frac{M_1 \|f\|_{L^1(\mathbb{R}^n)}}{y}\,\,\,\,\,\,\text{with norm}\,\, M_1=M_\varphi.
	\end{equation} 
	Then \eqref{vish5.3} will follows by the Marcinkiewicz interpolation theorem, cf. \cite[Theorem 1.1.3]{BL}. Now, to show \eqref{vish5.4}, using the Plancherel identity as in Proposition \ref{planch.A.ft} we get 
	\begin{align}
	\label{alsostrong}
	y^2 \nu_{\mathbb{N}}(y; Tf)&\leq  \|Tf\|^2_{L^2(\mathbb{N}, \nu)}= \sum_{j \in \mathbb{N}} \left( \frac{|\mathcal{F}_{A}(f)(j)|}{\varphi(j) {\|u_j\|_{L^\infty(\mathbb{R}^n)} } } \right)^2 \varphi^2(j) \|u_j\|_{L^\infty(\mathbb{R}^n)}^2 \nonumber  \\&= \sum_{j \in \mathbb{N}} |\mathcal{F}_{A}(f)(j)|^2 = \|\mathcal{F}_{A}(f)\|_{l^2(L)}^2= \|f\|_{L^2(\mathbb{R}^n)}^2.
	\end{align}
	Thus, $T$ is of weak type $(2,2)$ with norm $M_2 \leq 1.$ Further, we show that $T$ is of weak type $(1,1)$ with norm $M_1=M_\varphi$; more precisely, we show that 
	\begin{equation}
	\nu_{\mathbb{N}} \left\{j \in \mathbb{N}: \frac{|\mathcal{F}_{A}(f)(j)|}{\varphi(j) {\|u_j\|_{L^\infty(\mathbb{R}^n)} }}>y \right\} \lesssim M_\varphi \frac{\|f\|_{L^1(\mathbb{R}^n)}}{y}.
	\end{equation}
	Here, the left hand-side is the weighted sum $\sum \varphi^2(j) \|u_j\|_{L^\infty(\mathbb{R}^n)}^2 $ taken over those $j \in \mathbb{N}$ such that $\frac{|\mathcal{F}_{A}(f)(j)|}{\varphi(j) {\|u_j\|_{L^\infty(\mathbb{R}^n)} }}>y.$ From the definition of the Fourier transform and H\"{o}lder's inequality it follows that 
	$$|\mathcal{F}_{A}(f)(j)| \leq  \|u_j\|_{L^\infty(\mathbb{R}^n)} \|f\|_{L^1(\mathbb{R}^n)}.$$
	Therefore, we get 
	$$ y<\frac{|\mathcal{F}_{A}(f)(j)|}{\varphi(j) {\|u_j\|_{L^\infty(\mathbb{R}^n)} }} \leq \frac{\|f\|_{L^1(\mathbb{R}^n)}}{\varphi(j)  }.$$ 
	Using this, we get 
	$$ \left\{j \in \mathbb{N}: \frac{|\mathcal{F}_{A}(f)(j)|}{\varphi(j) {\|u_j\|_{L^\infty(\mathbb{R}^n)} }}>y \right\} \subset \left\{j \in \mathbb{N}: \frac{\|f\|_{L^1(\mathbb{R}^n)}}{\varphi(j)}>y \right\} $$ for any $y>0.$ Consequently, 
	$$\nu\left\{j \in \mathbb{N}: \frac{|\mathcal{F}_{A}(f)(j)|}{\varphi(j) {\|u_j\|_{L^\infty(\mathbb{R}^n)} }}>y \right\} \leq \nu \left\{j \in \mathbb{N}: \frac{\|f\|_{L^1(\mathbb{R}^n)}}{\varphi(j)}>y \right\}.$$
	By setting $w:= \frac{\|f\|_{L^1(\mathbb{R}^n)}}{y},$ we get 
	\begin{equation}
	\nu\left\{j \in \mathbb{N}: \frac{|\mathcal{F}_{A}(f)(j)|}{\varphi(j) }>y \right\} \nonumber
	\leq \sum_{\overset{j \in \mathbb{N}}{\varphi(j)  \leq w}} \varphi^2(j) \|u_j\|^2_{L^\infty(\mathbb{R}^n)} .
	\end{equation}
	We claim that 
	\begin{equation}\label{vish5.8}
	\sum_{\overset{j \in \mathbb{N}}{\varphi(j)  \leq w}} \varphi^2(j) \|u_j\|^2_{L^\infty(\mathbb{R}^n)}  \lesssim M_\varphi w.
	\end{equation}
	In fact, we have 
	$$ \sum_{\overset{j \in \mathbb{N}}{\varphi(j)  \leq w}} \varphi^2(j) \|u_j\|^2_{L^\infty(\mathbb{R}^n)}  =\sum_{\overset{j \in \mathbb{N}}{\varphi(j)  \leq w}}  \|u_j\|^2_{L^\infty(\mathbb{R}^n)}  \int_{0}^{\varphi^2(j)} d\tau.$$
	
	We can interchange sum and integration  to get 
	$$ \sum_{\overset{j \in \mathbb{N}}{\varphi(j)  \leq w}}  \|u_j\|^2_{L^\infty(\mathbb{R}^n)}  \int_{0}^{\varphi^2(j)} d\tau \leq \int_{0}^{w^2} d\tau \sum_{\overset{j \in \mathbb{N}}{\tau^{\frac{1}{2}} \leq \varphi(j)  \leq w}}  \|u_j\|^2_{L^\infty(\mathbb{R}^n)} .$$
	Further, we make a substitution $\tau =t^2,$ yielding
	$$\int_{0}^{w^2} d\tau \sum_{\overset{j \in \mathbb{N}}{\tau^{\frac{1}{2}} \leq \varphi(j)  \leq w}}  \|u_j\|^2_{L^\infty(\mathbb{R}^n)} =2 \int_{0}^{w } t dt \sum_{\overset{j \in \mathbb{N}}{t \leq \varphi(j)  \leq w}} \|u_j\|^2_{L^\infty(\mathbb{R}^n)} \leq 2 \int_{0}^{w } t \,dt \sum_{\overset{j \in \mathbb{N}}{t \leq \varphi(j)  }}  \|u_j\|^2_{L^\infty(\mathbb{R}^n)}.$$

	Now, since $$ t \sum_{\overset{j \in \mathbb{N}}{t \leq \varphi(j) }}  \|u_j\|^2_{L^\infty(\mathbb{R}^n)}   
	\leq \sup_{t>0} t  \sum_{\overset{j \in \mathbb{N}}{t \leq \varphi(j) }}  \|u_j\|^2_{L^\infty(\mathbb{R}^n)}   = M_\varphi$$
	is finite by assumption, we have
	$$   2 \int_{0}^{w } t dt \sum_{\overset{j \in \mathbb{N}}{t \leq \varphi(j) }}  \|u_j\|^2_{L^\infty(\mathbb{R}^n)}  \lesssim M_\varphi w= \frac{M_\varphi\|f\|_{L^1(\mathbb{R}^n)}}{y} . $$
	This proves \eqref{vish5.8}. Therefore, we have proved inequality \eqref{vish5.4} and \eqref{vish5.5}.  Then by using the Marcinkiewicz interpolation theorem we obtain that for $1<p<2$
	\begin{align*} 
	\left( \sum_{j \in \mathbb{N}} \left( \frac{|\mathcal{F}_{\mathcal{A}}(f)(j)|}{\|u_j\|_{L^\infty(\mathbb{R}^n)} \varphi(j)} \right)^p  \|u_j\|^2_{L^\infty(\mathbb{R}^n)}  \varphi(j)^{2}  \right)^{\frac{1}{p}}  = \|Tf\|_{L^p(\mathbb{N}, \nu)} \lesssim_{p}  M_\varphi^{\frac{2-p}{p}} \|f\|_{L^p(\mathbb{R}^n)}\,.
	\end{align*} 
	At the same time \eqref{alsostrong} suggest that the last inequality holds true for $p=2$ as an equality. The proof is now complete.

\end{proof}

The following theorem, see e.g. \cite{BL}, is an interpolation theorem of $L^p$ spaces with change of measure.

\begin{theorem} \label{interpolation} Let $d\mu_0(x)= \omega_0(x) d\mu(x),$ $d\mu_1(x)= \omega_1(x) d\mu(x)$. Let us also write $L^p(\omega)=L^p(\omega d\mu)$ for the weight $\omega.$ Suppose that $0<p_0, p_1< \infty.$ Then 
	$$(L^{p_0}(\omega_0), L^{p_1}(\omega_1))_{\theta, p}=L^p(\omega)\,,\quad 0<\theta<1\,,$$ where $\frac{1}{p}= \frac{1-\theta}{p_0}+\frac{\theta}{p_1}$ and $\omega= \omega_0^{\frac{p(1-\theta)}{p_0}} \omega_1^{\frac{p\theta}{p_1}}.$
\end{theorem} 

As a corollary of Theorem \ref{interpolation} we get the following extension (cf. [Corollary 5.5.4 \cite{BL}] ) of the interpolation theorem of Stein-Weiss.

\begin{corollary}[Stein-Weiss]\label{interpolationoperator} 	Assume that $1 \leq p_0,p_1,q_0,q_1<\infty$, and that 
	\[
	T:L^{p_0}(\omega_{0}d\mu)\rightarrow L^{q_0}(\tilde{\omega_0}d\nu)\,,
	\]
	\[
	T:L^{p_1}(\omega_1 d\mu)\rightarrow L^{q_1}(\tilde{\omega_1}d\nu)\,,
	\]
	with norms $M_0$ and $M_1$, respectively. Then 
	\[
	T:L^p(\omega d\mu)\rightarrow L^q(\tilde{\omega}d\nu)\,,
	\]
	with norm $M$ such that 
	\[
	 M \leq M_{0}^{1-\theta}M_{1}^{\theta}\,,
	\] 
	where 
	\[	
	\frac{1}{p}=\frac{1-\theta}{p_0}+\frac{\theta}{p_1}\,,\quad \frac{1}{q}=\frac{1-\theta}{q_0}+\frac{\theta}{q_1}\,,\]
	\[
	\omega=\omega_{0}^{p(1-\theta)/p_1}\omega_{1}^{p\theta/p_1}\,,\quad \tilde{\omega}=\tilde{\omega_{0}}^{q(1-\theta)/q_0}\tilde{\omega_{1}}^{q\theta/q_1}\,.
	\]
\end{corollary} 

Using the above, we can now prove a version of the Hausdorff-Young-Paley inequality adapted to our setting.

\begin{theorem}[Hausdorff-Young-Paley inequality] \label{HYP} Let $1<p\leq 2,$ and let   $1<p \leq b \leq p' \leq \infty,$ where $p'= \frac{p}{p-1}$.  If $\varphi(j)$ is a positive sequence in $\mathbb{N}$ such that 
	$$M_\varphi:= \sup_{t>0} t  \sum_{\overset{j \in \mathbb{N}}{t \leq \varphi(j) }}  \|u_j\|^2_{L^\infty(\mathbb{R}^n)}    $$
	is finite, then for every $f \in {L^p(\mathbb{R}^n)}$ 
	we have
	\begin{equation} \label{Vish5.9}
	\left( \sum_{j \in \mathbb{N}}  \left( |\mathcal{F}_{A}f(j)| \varphi(j)^{\frac{1}{b}-\frac{1}{p'}} \right)^b \|u_j\|_{L^\infty(\mathbb{R}^n)}^{2-b}       \right)^{\frac{1}{b}} \lesssim M_\varphi^{\frac{1}{b}-\frac{1}{p'}} \|f\|_{L^p(\mathbb{R}^n)}.
	\end{equation}
\end{theorem}
\begin{proof}
		By the Paley-type inequality as in Theorem \ref{PI}, the operator $T$ defined by 
	$$Tf(j):= \left\{ \frac{\mathcal{F}_{\mathcal{A}}f(j)}{\|u_j\|_{L^\infty(\mathbb{R}^n)} } \right\}_{j \in \mathbb{N}},$$ 
	is bounded from $L^p(\mathbb{R}^n)$ to $L^{p}(\mathbb{N},\omega_0),$ where $\omega_{0}(j)=\|u_j\|^2_{L^\infty(\mathbb{R}^n)}  \varphi(j)^{2-p}
	$, and the space $L^p(\mathbb{N},\omega)$ has been endowed with the norm $$\|\cdot\|_{L^p(\mathbb{N}, \omega)}:= \left( \sum_{j \in \mathbb{N}} |\cdot(j)|^p w(j)  \right)^{\frac{1}{p}}\,,\quad $$
	From Hausdroff-Young inequality as in Theorem \ref{haus}, we deduce that $T:L^p(\mathbb{R}^n) \rightarrow L^{p'}(\mathbb{N}, \omega_1)$ with   $\omega_1(j)= \|u_j\|^{2}_{L^\infty(\mathbb{R}^n)},$  admits a bounded extension, while, by Corollary \ref{interpolationoperator}, the operator $T:L^p(\mathbb{R}^n) \rightarrow L^{b}(\mathbb{N}, \omega),$ is continuous for $p \leq b\leq p^{'}$ such that $\frac{1}{b}=\frac{1-\theta}{p}+\frac{\theta}{p'}$, where $\theta \in (0,1)$. By fixing $\theta=\frac{p-b}{b(p-2)},$ we get 
		\begin{equation*}
	\omega= \omega_0^{\frac{p(1-\theta)}{p_0}} \omega_1^{\frac{p\theta}{p_1}}= \varphi(j)^{1-\frac{b}{p'}} \|u_j\|_{L^\infty(\mathbb{R}^n)}^{2-b}\,, 
	\end{equation*}
so that for the operator norm we have $\|T\|_{L^p \rightarrow L^b} \leq M_{\varphi}^{\frac{1}{b}-\frac{1}{p^{'}}}\,.$
	The proof is now complete.
 \end{proof}
Naturally, the Hausdorff-Young Paley inequality \eqref{Vish5.9} reduces to the Hausdorff-Young inequality \eqref{haus.1} when $b=p^{'}$ and to the Paley-type inequality \eqref{Vish5.1} when $b=p$.
\subsection{$L^p$-$L^q$ boundedness} 
In this paragraph we prove the $L^p$-$L^q$ boundedness of Fourier multipliers in our setting. Analogous results have been proved in the case of the torus in \cite{NT100} using a different method. We recall here that $m$ is an $A$-Fourier multiplier on $\mathbb{R}^n$ if it satisfies 
$$\mathcal{F}_{A}({mf})(j)= \sigma_{m}(j) \mathcal{F}_{A}{f}(j),\,\,\,\, j \in \mathbb{N},$$
 where $\sigma_{m} :\mathbb{N} \rightarrow \mathbb{C}$ is a function that shall be called the $A$-symbol, or simply symbol, of the multiplier $m$.
\begin{theorem}\label{Th:LpLq-1}
Let $1<p \leq 2 \leq q <\infty.$  Suppose that $m$ is an $A$-Fourier multiplier with symbol $\sigma_{m}$ on $\mathbb{R}^n.$ Then we have 
$$ \|m\|_{L^p(\mathbb{R}^n) \rightarrow L^q(\mathbb{R}^n)} \lesssim \sup_{s>0} s\left(  \ \sum_{\overset{j \in \mathbb{N}}{|\sigma_m(j)| >s}} \|u_j\|^2_{L^\infty(\mathbb{R}^n)}  \right)^{\frac{1}{p}-\frac{1}{q}}.$$
\end{theorem}
\begin{proof} 
Let us first assume that $p \leq q',$ where $\frac{1}{q}+\frac{1}{q'}=1.$ Since $q' \leq 2,$ the Hausdorff-Young inequality Theorem \ref{haus} gives that 
	\begin{align}\label{byhaus}
	\|mf\|_{L^q(\mathbb{R}^n)} & \leq  \|\mathcal{F}_{A}(mf)\|_{\ell^{q'}(A)}=\|\sigma_{m} \mathcal{F}_{A}(f)\|_{\ell^{q'}(A)}\nonumber\\& = \left( \sum_{j \in \mathbb{N}} |\sigma_{m}(j) |^{q'} |\mathcal{F}_{A}(f)(j)|^{q'} \|u_j\|^{2-q'}_{L^\infty(\mathbb{R}^n)} \right)^{\frac{1}{q'}}.
	\end{align}
	
	The case $q' \leq (p')'=p$ can be reduced to the case $p \leq q'$ as follows. Using the duality of $L^p$-spaces we have $\|m\|_{L^p(\mathbb{R}^n) \rightarrow L^q(\mathbb{R}^n)}= \|m^*\|_{L^{q'}(\mathbb{R}^n) \rightarrow L^{p'}(\mathbb{R}^n)}.$ The symbol $\sigma_{m^*}(j)$ of the adjoint operator $m^*,$ which is also an $A$-Fourier multiplier, is equal to $\overline{\sigma_{m}(j)}$ (cf. [Proposition 9.7 \cite{Ruz-Tok}]), and obviously we have $|\sigma_{m}(j)|= |\sigma_{m^*}(j)|$. Therefore \eqref{byhaus} holds true for $1<p \leq 2 \leq q <\infty$.
	Now, we are in a position to apply Hausdorff-Young-Paley inequality Theorem \ref{HYP}. Set $\frac{1}{p}-\frac{1}{q}=\frac{1}{r}.$ Then, applying Theorem \ref{HYP} with $\varphi(j)= |\sigma_{m}(j)|^r$ with $b=q'$ we get 
	$$\|\sigma_{m}\mathcal{F}_{A}(f)\|_{\ell^{q'}(A)} \lesssim \left(  \sup_{s>0} s \sum_{\overset{j \in \mathbb{N}}{|\sigma_m(j)|^r>s}} \|u_j\|^2_{L^\infty(\mathbb{R}^n)}   \right)^{\frac{1}{r}} \|f\|_{L^p(\mathbb{R}^n)} $$ for all $f \in L^p(\mathbb{R}^n),$ in view of $\frac{1}{p}-\frac{1}{q}=\frac{1}{q'}-\frac{1}{p'}=\frac{1}{r.}$ Thus, for $1<p \leq 2 \leq q<\infty,$ we obtain 
	
	$$\|mf\|_{L^q(\mathbb{R}^n)} \lesssim \left(  \sup_{s>0} s \sum_{\overset{j \in \mathbb{N}}{|\sigma_m(j)|^r>s}} \|u_j\|^2_{L^\infty(\mathbb{R}^n)}  \right)^{\frac{1}{r}} \|f\|_{L^p(\mathbb{R}^n)}\,, $$
	and this completes the proof Theorem \ref{Th:LpLq-1} under the observation that
	\begin{align*}
	& \left(  \sup_{s>0} s \sum_{\overset{j \in \mathbb{N}}{ |\sigma_m(j)|^r>s}}\|u_j\|^2_{L^\infty(\mathbb{R}^n)}  \right)^{\frac{1}{r}} =  \left(  \sup_{s>0} s \sum_{\overset{j \in \mathbb{N}}{|\sigma_m(j)|  >s^{\frac{1}{r}}}} \|u_j\|^2_{L^\infty(\mathbb{R}^n)}  \right)^{\frac{1}{r}} \\&= \left(  \sup_{s>0} s^r \sum_{\overset{j \in \mathbb{N}}{|\sigma_m(j)| >s}} \|u_j\|^2_{L^\infty(\mathbb{R}^n)}  \right)^{\frac{1}{r}} = \sup_{s>0} s\left(  \ \sum_{\overset{j \in \mathbb{N}}{|\sigma_m(j)| >s}} \|u_j\|^2_{L^\infty(\mathbb{R}^n)}  \right)^{\frac{1}{r}}\,.
	\end{align*} 
  \end{proof}
  
 
 The following corollaries present  $L^p$-$L^q$ boundedness results for particular cases of spectral multipliers associated with the anharmonic oscillator. 
 
 \begin{corollary}\label{cor.anh.1}
  Let $1<p\leq 2 \leq q<\infty$, and let $\varphi$ be any decreasing function on $[1,\infty)$ such that
 $\lim_{u \rightarrow \infty} \varphi(u)=0$. Then,
  \begin{equation}
  \label{cor.1.estim}
  \|\varphi(A)\|_{op} \lesssim \sup_{u>1} \varphi(u)\left(u^{1+\frac{(k+l)n}{2kl}}\right)^{\frac{1}{p}-\frac{1}{q}}\,, 
  \end{equation}
   where  $\|\cdot \|_{op}$ denotes the operator norm from $L^{p}(\mathbb{R}^n)$ to $L^{q}(\mathbb{R}^n)$.
 \end{corollary}

 \begin{proof}
 	Observe that the operator $\varphi(A)$ is a Fourier multiplier with symbol $(\varphi(\lambda_j))_{j \in \mathbb{N}}$. Therefore an application of Theorem \ref{Th:LpLq-1} yields
	
	\begin{align}
	\label{first.estim.first.cor}
	\|\varphi(A)\|_{op}  \lesssim  \sup_{s>0} s \left(\sum_{\overset{j \in \mathbb{N}}{\varphi(\lambda_j)>s}} \|u_j\|^{2}_{L^{\infty}(\mathbb{R}^n)} \right)^{\frac{1}{p}-\frac{1}{q}} = \sup_{0<s\leq \varphi(1)} s \left( \sum_{\overset{j \in \mathbb{N}}{\varphi(\lambda_j)>s}}\|u_j\|^{2}_{L^{\infty}(\mathbb{R}^n)}\right)^{\frac{1}{p}-\frac{1}{q}} \end{align}
	since $\varphi \leq \varphi(1)$, so that $s \in (0,\varphi(1)]$ can be written as $s=\varphi(u)$. Now, by Lemma \ref{Sob.emb} together with \eqref{first.estim.first.cor} and by using the monotonicity of $\varphi$ we get
	\begin{align}\label{th.for.u}
	\|\varphi(A)\|_{op} \lesssim  \sup_{u>1}  \varphi(u) \left(\sum_{\overset{j \in \mathbb{N}}{\varphi(\lambda_j)>\varphi(u)}} \|u_j\|^{2}_{L^{\infty}(\mathbb{R}^n)} \right)^{\frac{1}{p}-\frac{1}{q}}\lesssim  \sup_{u>1} \varphi(u) \left( \sum_{\overset{j \in \mathbb{N}}{\lambda_j <u }} \lambda_{j}\right)^{\frac{1}{p}-\frac{1}{q}} \,.
	\end{align}
	Now, notice that as $u \rightarrow \infty$, there exists some $j_0 \in \mathbb{N}$ big enough such that $u \sim \lambda_{j_{0}}$. Then one can estimate
	\begin{equation}\label{estimatewithN}
	\sum_{\overset{j \in \mathbb{N}}{\lambda_{j}<u}}\lambda_{j} \lesssim u \sum_{\overset{j \in \mathbb{N}}{\lambda_j < \lambda_{j_0}}}1 \lesssim u N(\lambda_{j_0})\lesssim u \lambda_{j_0}^{\frac{(k+l)n}{2kl}} \sim u^{1+\frac{(k+l)n}{2kl}}\,,
	\end{equation}
	where the eigenvalue counting function $N(\lambda_{j_0})$ associated to $A$ is as in \eqref{eig.count.funt.anhar.}. Therefore, by combining \eqref{th.for.u} with \eqref{estimatewithN} one gets 
	\[
	\|\varphi(A)\|_{op}\lesssim \sup_{u>1} \varphi(u) \left(u^{1+\frac{(k+l)n}{2kl}}\right)^{\frac{1}{p}-\frac{1}{q}}\,,
	\]
	and the last completes the proof of Corollary \ref{cor.anh.1}.
 \end{proof}
 
 \begin{remark}
 Notice that one can estimate the operator norm $\|\varphi(A)\|_{op}$ with $\varphi$ as in Corollary \ref{cor.anh.1} just by using the fact that $A^{-1} \in S_{r}(L^2(\mathbb{R}^n))$ for any $r>\frac{(k+l)n}{2kl}$. Indeed, by using that $u \sim \lambda
 _{j_0}$ for some $j_0$ big enough, while also by the eigenvalue asymptotics \eqref{eigen.vanish} we have that $j^{\frac{1}{r}} \lesssim \lambda_j$ for any  $r>\frac{(k+l)n}{2kl}$, one can alternatively get the following estimate
 \begin{align}
    \label{eig.shatten}
    \sum_{\overset{j \in \mathbb{N}}{\lambda_j<u}}\lambda_{j}^{2} \lesssim \sum_{\overset{j \in \mathbb{N}}{\lambda_j<\lambda_{j_{0}}}}\lambda_{j}^{2}\leq \lambda_{j_0}^2 \sum_{j=1}^{j_0}1=j_0 \lambda_{j_0}^{2} \lesssim \lambda_{j_0}^{2} \lambda_{j_0}^{r} \sim u^{2+r}  \,.
\end{align}
A combination of \eqref{th.for.u} together with \eqref{eig.shatten} yields
\begin{equation}\label{alt.estimat}
\|\varphi(A)\|_{op} \lesssim \sup_{u>1}\varphi(u)\left(u^{2+r}\right)^{\frac{1}{p}-\frac{1}{q}}\,.
\end{equation}
Observe that the estimate  \eqref{cor.1.estim} is sharper than the one given by \eqref{alt.estimat}.
 \end{remark}

 \begin{corollary}\label{bound.inverse}
  The operator $A^{-m}$ for $m> \left(1+\frac{(k+l)n}{2kl}\right) \left(\frac{1}{p}-\frac{1}{q}\right)$ is bounded from $L^p(\mathbb{R}^n)$ to $L^q(\mathbb{R}^n)$. In particular, the operator $A^{-m}$ is bounded from $L^2(\mathbb{R}^n)$ to $L^2(\mathbb{R}^n)$ for any $m \geq 0$.
 \end{corollary}
 \begin{proof}
  Let $\varphi=\varphi(u):=u^{-m}$, for $u \geq 1$ and $m>0$. Then, $\varphi$ satisfies the hypothesis of Corollary \ref{cor.anh.1} and we have 
  \[
  \|A^{-m}\|_{L^p(\mathbb{R}^n)\rightarrow L^q(\mathbb{R}^n)} \lesssim \sup_{u>1}u^{-m} u^{C_{n,k,l}\left(\frac{1}{p}-\frac{1}{q} \right)}<\infty\,,
  \]
  for $C_{n,k,l}\left(\frac{1}{p}-\frac{1}{q} \right) \leq m$, where $C_{n,k,l}=1+\frac{(k+l)n}{2kl}$. 
 \end{proof}
 
 \begin{remark}\label{Sobointro}[Sobolev-type estimates]
	If we choose \[\varphi=\varphi(u):=\frac{1}{u^{a-b}}\,,\quad u\geq 1\,,\]
	then $\varphi$ satisfies the hypothesis of Corollary \ref{cor.anh.1}, and so for $f \in \mathcal{S}(\mathbb{R}^n)$ we have 
	\[
	\|A^{-a}{A}^{b}f\|_{L^q(\mathbb{R}^n)} \lesssim \|f\|_{L^p(\mathbb{R}^n)}\,,
	\]
	for $a,b$ being such that $a-b\geq C_{n,l,k} \cdot \left(\frac{1}{p}-\frac{1}{q}\right)^{-1}$, where $C_{n,l,k}=1+\frac{(k+l)n}{2kl}$, and for $1<p\leq 2 \leq q < \infty$, or equivalently 
	\begin{equation}\label{a-b}
	\|A^bf\|_{L^q(\mathbb{R}^n)} \lesssim \|A^{a}f\|_{L^p(\mathbb{R}^n)}\,.
	\end{equation}
	As particular cases of \eqref{a-b} we get:
	\begin{enumerate}
		\item For $p=q=2$ we recover the inclusion
			\[
		\|f\|_{\mathcal{H}^{b}}\lesssim \|f\|_{\mathcal{H}^{a}}\,,\quad a \geq b\,,
		\]	
		where the Sobolev norms $\|\cdot\|_{\mathcal{H}^m}$ have been defined in \eqref{Sob.spaces.anh}. 
		\item For $a\geq C_{n,l,k} \cdot \left(\frac{1}{p}-\frac{1}{q}\right)^{-1}$, $b=0$, we get the Sobolev-type estimate
		\[
		\|f\|_{L^q(\mathbb{R}^n)}\leq C \|A^{a}f\|_{L^p(\mathbb{R}^n)}\,.
		\]	
	\end{enumerate}
\end{remark}
 
 \begin{remark} \label{Heatintro}[The $A$-heat equation]
	For the operator $A$ we consider the heat equation 
	\begin{equation}
	\label{heat.eq}
	\partial_{t}v+Av=0\,,\quad v(0)=v_0\,.
	\end{equation}
	One can check that for each $t>0$ the function $v(t)=v(t,x):=e^{-tA}v_0$ is a solution of the initial value problem \eqref{heat.eq}. For $t>0$, consider the function $\varphi=\varphi(u):=e^{-tu}$. Now, $\varphi$ satisfies the assumptions of Corollary \ref{cor.anh.1} and we get 
	\begin{equation}
	\label{heat.eq1}
	\|e^{-tA}\|_{L^{p}(\mathbb{R}^n)\rightarrow L^q(\mathbb{R}^n)} \lesssim \sup_{u>0} e^{-tu}\cdot u^{-C_{n,k,l}\left(\frac{1}{p}-\frac{1}{q}\right)}\,,
	\end{equation}
	for $1<p\leq 2\leq q <\infty$ and $C_{n,k,l}=1+\frac{(k+l)n}{2kl}.$ Using techniques of standard mathematical analysis we see that 
	\begin{eqnarray}
\label{heat.estimat.}
\sup_{u>0}e^{-tu} \cdot u^{C_{n,k,l}\left(\frac{1}{p}-\frac{1}{q}\right)}& = & \left( \frac{C_{n,k,l}}{t}\left(\frac{1}{p}-\frac{1}{q}\right) \right)^{C_{n,k,l}\left( \frac{1}{p}-\frac{1}{q}\right)}\cdot e^{-C_{n,k,l}\left(\frac{1}{p}-\frac{1}{q}\right)}\nonumber\\
& =& C_{p,q}t^{-C_{n,k,l}\left(\frac{1}{p}-\frac{1}{q} \right)}\,.
\end{eqnarray}
	Indeed, let us consider the function $g(u)=e^{-tu} u^{C_{n,k,l}\left(\frac{1}{p}-\frac{1}{q} \right)}$, and compute its derivative 
	\[
	g^{'}(u)=u^{C_{n,k,l}\left(\frac{1}{p}-\frac{1}{q} \right)-1}e^{-tu} \left( C_{n,k,l}\left(\frac{1}{p}-\frac{1}{q} \right)-ut \right)\,.
	\]
	The only zero of the derivative $g^{'}$ is taken for $u_0= \frac{C_{n,k,l}}{t}\left(\frac{1}{p}-\frac{1}{q}\right)$ and $g^{'}$ changes sign from positive to negative at $u_0$.  Thus, $g$ attends a maximum at $u_0$ and we have proved \eqref{heat.estimat.}.
	Finally, combining \eqref{heat.eq1} with \eqref{heat.estimat.} we obtain 
	\[
	\|v(t,\cdot)\|_{L^q(\mathbb{R}^n)} \lesssim C_{p,q}t^{-C_{n,k,l}\left(\frac{1}{p}-\frac{1}{q} \right)}\ \|v_0\|_{L^q(\mathbb{R}^n)}\,,
	\]
	where $C_{n,k,l}=1+\frac{(k+l)n}{2kl}$.
\end{remark}

\section{Nuclearity and traces of spectral multipliers of anharmonic oscillator on modulation spaces}
This section is devoted to the study of $r$-nuclearity, where $0<r \leq 1$, and traces of spectral multipliers of the anharmonic oscillator $A.$ 

First we give a short exposition of nuclear operators on Banach spaces with the metric approximation property. Later on, we apply these ideas to the study of our operators when acting on \textit{modulation spaces} $\mathcal{M}_{w}^{p,q}$, where $1\leq p,q <\infty$, on $\mathbb{R}^n$.

Let $T\in \mathcal{L}(\mathcal{B},\mathcal{B})$ be a linear operator from some Banach space $\mathcal{B}$ to $\mathcal{B}$. We say that $T$ is a \textit{nuclear operator on $\mathcal{B}$}, and we write $T \in \mathcal{N}(\mathcal{B})$, if $T$ admits a decomposition of the form
\begin{equation}
\label{inf.repr.}
T=\sum_{j=1}^{\infty}\varphi_j \otimes\psi'_j\,, \quad \text{where}\quad (\varphi_j)_{j=1}^{\infty} \subset \mathcal{B}\quad \text{and}\quad (\psi'_j)_{j=1}^{\infty} \subset \mathcal{B}^{'}\,,
\end{equation}
and we have $\sum_{j=1}^{\infty} \|\psi'_j\|_{\mathcal{B}^{'}}\|\varphi_j\|_{\mathcal{B}}<\infty$. 

Grothendieck in \cite{Groth2} proved that if $\mathcal{B}$ has the metric approximation property, and $T \in \mathcal{N}(\mathcal{B})$, then the trace of $T$, denoted by $\textnormal{Tr}(T)$, is well defined, i.e., we have that
\begin{equation}
\label{tr.nucl.1}
\textnormal{Tr}(T)=\sum_{j=1}^{\infty} \langle \varphi_j,\psi'_j \rangle_{\mathcal{B},\mathcal{B}^{'}}<\infty\,,
\end{equation}
Recall that the Banach space $\mathcal{B}$ has the metric approximation property if for every compact set $K \subset \mathcal{B}$ and $\epsilon>0$, there exists an operator $F$ of finite rank such that $\|x-Fx\|_{\mathcal{B}}<\epsilon$ for all $x \in K$, see [Lemma 10.2.20 \cite{Piet}] for this characterisation of Banach spaces with the metric approximation property.

\indent We note that nuclear operators agree with trace class operators in the setting of Hilbert spaces. For such operators Lidski\u{i} \cite{Lid} proved that
\begin{equation}
\label{Lid.form}
\textnormal{Tr}(T)=\sum_{j=1}^{\infty} \lambda_j\,,
\end{equation}
 where each eigenvalue $\lambda_j$ is repeated accordingly to multiplicity. Formula \eqref{Lid.form} is known as Lidski\u{i}'s formula. However, as we will see later, Liski\u i's formula can still hold true for operators on some Banach spaces, under certain assumptions. To become more precise, let us first recall the following definition.  
 
 Let $\mathcal{B}$ be a Banach space and let $0<r \leq 1$. A a linear operator $T \in \mathcal{L}(\mathcal{B},\mathcal{B})$ is called \textit{$r$-nuclear} if $T$ has a representation as in \eqref{inf.repr.} that satisfies
 	\[
 	\sum_{j=1}^{\infty}\|\psi'_j\|_{\mathcal{B}^{'}}^{r}\|\varphi_j\|_{\mathcal{B}}^{r}<\infty\,.
 	\]
 In \cite{Groth2}, Grothendieck proved that if $T \in \mathcal{L}(\mathcal{B},\mathcal{B})$ is a $\frac{2}{3}$-nuclear operator (and so also $r^{'}$-nuclear, with $r^{'}\leq 2/3$), and $\mathcal{B}$ has the metric approximation property, then the trace of $T$ can be calculated using  Lidski\u{i}'s formula \eqref{Lid.form}.

 The operators we consider here, act on the modulation spaces $\mathcal{M}_{w}^{p,q}$ on $\mathbb{R}^n$ introduced by  Feichtinger in \cite{F}. These spaces are, under some conditions on the weight function $w$, Banach spaces with the metric approximation property. Roughly speaking a modulation space is defined by imposing a quasi-norm estimate on the short-time Fourier transform (STFT) $V_gf$ of the involved distributions $f$. Formally, we give the following definition
\begin{definition}
	Let $1 \leq p,q<\infty$. For a suitable weight $w$ on $\mathbb{R}^{2n}$, and a window function $g \in \mathcal{S}(\mathbb{R}^n)$, we define the modulation space $\mathcal{M}_{w}^{p,q}(\mathbb{R}^n)$ to be the set of tempered distributions $f \in \mathcal{S}^{'}(\mathbb{R}^n)$ such that 
	\[
	\|f\|_{\mathcal{M}_{w}^{p,q}}:= \left( \int_{\mathbb{R}^n} \left( |V_gf(x,\xi)|^p w(x,\xi)^p\,dx\right)^{\frac{p}{q}}d\xi\right)^{\frac{1}{p}}<\infty\,,
	\]
	where $V_gf$ denotes the short-time Fourier transform of $f$ with respect to $g$ at the point $(x,\xi)$, i.e., we can write
	\[
	V_gf(x,\xi)=\int_{\mathbb{R}^n}f(y)\overline{g(y-x)}e^{-iy\cdot \xi}\,dy\,.
	\]
\end{definition}
The modulation space $\mathcal{M}_{w}^{p,q}(\mathbb{R}^n)$ endowed with the above norm becomes a Banach space. 

We recall that a weight function is a non-negative, locally integrable function on $\mathbb{R}^{2n}$. A weight function $u$ is called \textit{submultiplicative} if 
\[
u(x_1+x_2) \leq u(x_1)u(x_2)\,,\quad \text{for all}\quad x_1,x_2 \in \mathbb{R}^{2n}\,;
\]
a weight function $v$ is called $u$-\textit{moderate} if
\[
v(x_1+x_2) \leq u(x_1)v(x_2)\,,\quad \text{for all}\quad x_1,x_2 \in \mathbb{R}^{2n}\,.
\]
Weights $v_s$, $s \in \mathbb{R}$, of polynomial type; that is weights of the form 
\[
v_s(x,\xi)=(1+|x|^2+|\xi|^2)^{\frac{s}{2}}\,,
\]
play an important role. Any $v_s$-moderated weight function for some $s$ will be called \textit{polynomially moderated}.

 Polynomially moderate weights give rise to modulation spaces with the metric approximation property. In particular, we have the following Corollary as in [Corollary 3.1\cite{DRW}].
\begin{corollary}\label{mod.spac.poly.w.}
	Let $1 \leq p,q <\infty$, and let $w$ be a submultiplicative polynomially moderate weight. Then $\mathcal{M}_{w}^{p,q}(\mathbb{R}^n)$ has the metric approximation property.
\end{corollary}
Thus, for $T \in \mathcal{N}(\mathcal{M}_{w}^{p,w})$, where $T=\sum_{j=1}^{\infty}\varphi_j \otimes\psi'_j$, and $ (\varphi_j)_j \subset \mathcal{M}_{w}^{p,q}\,, (\psi'_j)_j \subset \mathcal{M}_{w^{-1}}^{p^{'},q^{'}}$ we can write
\[
\textnormal{Tr}(T)=\sum_{j=1}^{\infty} \langle \phi_j,\psi'_j \rangle_{\mathcal{M}_{w}^{p,q},\mathcal{M}_{w^{-1}}^{p^{'},q^{'}}}\,,
\]
where the duality has been defined via 
\begin{equation}\label{duality}
\langle f,h \rangle_{\mathcal{M}_{w}^{p,q},\mathcal{M}_{w^{-1}}^{p^{'},q^{'}}}=\int_{\mathbb{R}^{2n}}V_gf(x,\xi) \overline{V_g\overline{h}(x,\xi)}\,dx\,d\xi\,,
\end{equation}
for $f \in \mathcal{M}_{w}^{p,q}$, $h \in \mathcal{M}_{w^{-1}}^{p^{'},q^{'}}=\left(\mathcal{M}_{w}^{p,q} \right)^{'}$.\\

Note that the duality \eqref{duality} is different, with respect to taking the complex conjugate, from the one in \cite{Groch} to make it compatible with the Grothendieck's theory.\\

\indent The main result of this chapter is an application of the next corollary as in [Corollary 5.1 \cite{DRW}].
\begin{corollary}
	\label{sec5.cor.}
	Let $0<r \leq 1$, $1\leq p,q<\infty$ and $w$ be a submultiplicative polynomially moderate weight. An operator $T \in \mathcal{L}(\mathcal{M}_{w}^{p,q},\mathcal{M}_{w}^{p,q})$ is $r$-nuclear if and only if its kernel $k(x,y)$ can be written in the form 
	\[
	k(x,y)=\sum_{j=1}^{\infty} \phi_j \otimes \psi'_j\,,
	\]
	with $\phi_j \in \mathcal{M}_{w}^{p,q}\,,\psi'_j \in \mathcal{M}_{w^{-1}}^{p^{'},q^{'}}$ and 
	\[
	\sum_{j=1}^{\infty}\|\phi_j\|_{\mathcal{M}_{w}^{p,q}}^{r}\|\psi'_j\|_{\mathcal{M}_{w^{-1}}^{p^{'},q^{'}}}^{r}<\infty\,.
	\]
	Moreover, if $T$ is $r$-nuclear with $r\leq \frac{2}{3}$, then 
	\begin{equation}
	\label{cor.trace.f}
	\textnormal{Tr}(T)=\sum_{j=1}^{\infty}\lambda_j\,,
	\end{equation}
	where $\lambda_j$, $j=1,2,\cdots$, are the eigenvalues of $T$ repeated according to multiplicity.
\end{corollary}
We note that in the context of a general Banach space with the approximation property, the range of $r$, $r \leq \frac{2}{3}$, in the $r$-nuclearity is sharp for the trace formula \eqref{cor.trace.f}. However, the above condition on $r$ can be relaxed if one considers for instance traces in the $L^p$ spaces (cf. \cite{DR}, \cite{RL}). In addition, as the next theorem shows, the trace formula \eqref{cor.trace.f} can still holds true in the context of an operator $f(A) \in \mathcal{L}(\mathcal{M}_{w}^{p,q}, \mathcal{M}_{w}^{p,q})$ for even larger values of $r$, and in particular even for simply nuclear operators (case $r=1$).

 We can write with convergence in $L^2(\mathbb{R}^n)$
\[
g=\sum_{j=1}^{\infty}(g,\overline{u_j})_{L^2}u_j\,,
\]
where $(u_j)_j \subset L^2(\mathbb{R}^n)$ is the set of eigenfunctions of $A$.

Hence, the kernel of $A$ can be written as
\[
k(x,y)=\sum_{j=1}^{\infty}A u_j(x)\overline{u}_j(y)=\sum_{j=1}^{\infty}\lambda_j u_j(x)\overline{u}_j(y)\,.
\]
This can be justified by taking $N>0$ large enough so that $\sum_{j=1}^{\infty} \lambda_{j}^{-N}<\infty$ and the kernel of $A^{-N}$ can be decomposed as 
\[
k_{A^{-N}}(x,y)=\sum_{j=1}^{\infty}\lambda_{j}^{-N}u_j(x) \overline{u}_j(y)\,.
\]
For functions $f$ of $A$, the kernel can be expressed in the form 
\[
k_{f(A)}=\sum_{j=1}^{\infty}f(\lambda_j)u_j(x)\overline{u}_j(y)\,.
\]
\begin{theorem}\label{delgado} 
	Let $0<r \leq 1$, $1\leq p,q<\infty$ and $w$ be a submultiplicative polynomially moderate weight. The operator $f(A)$ is $r$-nuclear on $\mathcal{M}_{w}^{p,q}(\mathbb{R}^n)$, provided that 
	\begin{equation}\label{sec5.thm.as.1}
	\sum_{j=1}^{\infty}|f(\lambda_j)|^r \|u_j\|^{r}_{\mathcal{M}_{w}^{p,q}}\|u_{j}\|^{r}_{\mathcal{M}_{w^{-1}}^{p^{'},q^{'}}}<\infty\,.
	\end{equation}
	If, in particular, \eqref{sec5.thm.as.1} holds for $r=1$, then we have the trace formula
	\begin{equation}
	\label{sec5.thm.as.2}
	\textnormal{Tr}(f(A))= \sum_{j=1}^{\infty} f(\lambda_j)\,,
	\end{equation}
	where the series $\sum_{j=1}^{\infty} f(\lambda_j)$ converges absolutely.
\end{theorem}
We note that the modulation space $\mathcal{M}_{w}^{p,q}$ corresponds to a weight of the form $w(x,\xi)=(1+|x|^2+|\xi|^2)^{\frac{s}{2}}$, while the STFT of $f$ with respect to $g$, $V_gf$ is taken with respect to the standard Gaussian window $g(x)=2^{\frac{n}{4}}e^{-\pi x^2}$.
\begin{proof}
	The first part follows from Corollary \ref{sec5.cor.}. Formula \eqref{sec5.thm.as.2} is expected from the general Grothendieck's theory in the case where $r \leq \frac{2}{3}$. For $r=1$, by \eqref{tr.nucl.1} we obtain \[
	\textnormal{Tr}(f(A))=\sum_{j=1}^{\infty}f(\lambda_j) \langle u_j,\overline{u}_j \rangle_{\mathcal{M}_{w}^{p,q},\mathcal{M}_{w^{-1}}^{p^{'},q^{'}}}=\sum_{j=1}^{\infty}f(\lambda_j)\,,
	\]
	since by the duality \eqref{duality} we have
	\[
	\langle u_j,\overline{u}_j \rangle_{\mathcal{M}_{w}^{p,q},\mathcal{M}_{w^{-1}}^{p^{'},q^{'}}}=( V_g u_j,V_gu_j )_{L^2}=( u_j,u_j)_{L^2}\|g\|_{L^2}= 1\,,
	\]
	where $( u_j,u_j)_{L^2}=\|g\|_{L^2}=1$, where we have used the orthogonality relations for the STFT, i.e., that
	\[
	( V_{g_1}f_1,V_{g_2}f_2)_{L^2}=( f_1,f_2)_{L^2} \overline{( g_1,g_2 )}_{L^2}\,,\quad \text{for}\quad f_1,f_2,g_1,g_2 \in L^2\,.
	\]
	The series in \eqref{sec5.thm.as.2} converges absolutely in view of 
	\begin{eqnarray*}
		\sum_{j=1}^{\infty}|f(\lambda_j)|& =& \sum_{j=1}^{\infty}|f(\lambda_j)| \langle u_j,u_j\rangle_{L^2}= \sum_{j=1}^{\infty}|f(\lambda_j)| \langle u_j,\overline{u}_j \rangle_{\mathcal{M}_{w}^{p,q},\mathcal{M}_{w^{-1}}^{p^{'},q^{'}}}\\
		& \leq & \sum_{j=1}^{\infty}|f(\lambda_j)| \|u_j\|_{\mathcal{M}_{w}^{p,q}}\|u_j\|_{\mathcal{M}_{w^{-1}}^{p^{'},q^{'}}}<\infty\,,
	\end{eqnarray*}
	which is finite by assumption. This completes the proof. 
\end{proof}
Observe that the proof of Theorem \ref{delgado} relies exclusively on the fact that the operator $A^{-1}$ is compact and self-adjoint on $L^2(\mathbb{R}^n)$. The following corollary is then immediate.
\begin{corollary}
	Let $0<r \leq 1$, $1 \leq p,q <\infty$ and $w$ be a submultiplicative polynomially weight. Let also $T$ be an operator on $L^2(\mathbb{R}^n)$ with discrete spectrum, and whose eigenvalues form an orthonormal basis in $L^2(\mathbb{R}^n)$. Then the operator $f(T)$ is $r$-nuclear on $\mathcal{M}_{w}^{p,q}(\mathbb{R}^n)$ provided that 
	\begin{equation}
	\label{cor.nucl.extra}
	\sum_{j=1}^{\infty}|f(\lambda_j)|^r \|u_j\|^{r}_{\mathcal{M}_{w}^{p,q}}\|u_j\|^{r}_{\mathcal{M}_{w^{-1}}^{p^{'},q^{'}}}<\infty\,,
	\end{equation}
	where $\{u_j\}_{j \in \mathbb{N}}$, $\{\lambda_j\}_{j \in \mathbb{N}}$ denote the sets of the eigefunctions and eigenvalues of $T$, respectively. If in particular, \eqref{cor.nucl.extra} holds for $r=1$, then we have the trace formula
	\[
	\textnormal{Tr}((f(T)))=\sum_{j=1}^{\infty}f(\lambda_j)\,,
	\]
	where the above series converges absolutely.
\end{corollary}

\bibliographystyle{amsplain}

\end{document}